\documentclass[letterpaper, 11pt]{amsart}
\usepackage{mathtools}
\usepackage{amsmath}
\usepackage[T1]{fontenc}
\usepackage{amssymb}
\usepackage{yhmath}
\usepackage{comment}
\usepackage{graphicx} 
\usepackage{ mathrsfs }
\usepackage{bbm} 
\usepackage{xcolor}
\usepackage{tikz-cd}
\usepackage{tikz}
\usetikzlibrary{patterns}
\usepackage{hyperref}

\DeclareMathAlphabet{\mathpzc}{OT1}{pzc}{m}{it}

\usepackage{thmtools}
\usepackage{thm-restate}

\usepackage{caption}

\newtheorem{theorem}{Theorem}[section]

\newtheorem*{claim*}{Claim}

\newtheorem{lemma}[theorem]{Lemma}

\newtheorem{corollary}[theorem]{Corollary}

\newtheorem{proposition}[theorem]{Proposition}

\theoremstyle{definition}
\newtheorem{definition}[theorem]{Definition}

\newtheorem{remark}[theorem]{Remark}

\numberwithin{equation}{section}

\DeclareMathOperator{\supp}{supp}

\DeclareMathOperator{\dist}{d}

\DeclareMathOperator{\Bc}{\mathcal{B}}
\DeclareMathOperator{\Cc}{\mathcal{C}}

\DeclareMathOperator{\Kc}{\mathcal{K}}
\DeclareMathOperator{\Lc}{\mathcal{L}}
\renewcommand{\Mc}{\mathcal{M}}

\DeclareMathOperator{\Tc}{\mathcal{T}}

\DeclareMathOperator{\Pc}{\mathcal{P}}

\DeclareMathOperator{\Xc}{\mathcal{X}}
\newcommand{\Ec}{\mathcal{E}}

\DeclareMathOperator{\Hb}{\mathbb{H}}

\DeclareMathOperator{\Nb}{\mathbb{N}}

\DeclareMathOperator{\Rb}{\mathbb{R}}

\DeclareMathOperator{\PO}{\mathsf{PO}}

\newcommand{\abs}[1]{\left|#1\right|}
\newcommand{\norm}[1]{\left\|#1\right\|}

\newcommand{\Ga}{\Gamma}

\newcommand{\CAT}{\operatorname{CAT}}

\renewcommand{\PO}{\mathsf{PO}}

\begin{document}

\title{Extremal entropy for products~of~Fuchsian~representations}

\author{Richard Canary}
\address{University of Michigan}
\thanks{Canary was partially supported by grant DMS-2304636 from the National Science Foundation and a Fellowship form the Simons Foundation. 
This paper is based upon work supported by the National Science Foundation
under Grant No. DMS-2424139, while the authors were in residence at the Simons Laufer Mathematical Sciences Institute during the Spring 2026 semester.}

\author{Dongryul M. Kim}
\address{
  Institute for Advanced Study (Current)
  \newline \indent Simons Laufer Mathematical Sciences Institute}

\date{\today}
\keywords{}
\subjclass[2020]{57K20, 37D40}

\begin{abstract}
In this paper, we count the number of  (almost) extremally stretched closed geodesics for a pair of (non-conjugate) Fuchsian representations of a closed surface group, and show that it has subexponential growth. We then deduce that, as a discrete subgroup of $\mathsf{PO}(2, 1) \times \mathsf{PO}(2, 1)$, the growth indicator of the product representation vanishes on the boundary of the Benoist limit cone. We also prove that for a general Zariski dense Borel Anosov subgroup of $\mathsf{PO}(2, 1) \times \mathsf{PO}(2, 1)$, its growth indicator vanishes on at least one boundary component of the Benoist limit cone, but not necessarily on both. One may view our first result as a sharpening of Thurston's result that there is a unique geodesic lamination  $\lambda$ such that every measured lamination maximizing the ratio of lengths with respect to the two representations has support contained in $\lambda$. We hope this will be a starting point for a more general study of extremal entropy.

\end{abstract}

\maketitle

\section{Introduction}

Let $S$ be a closed connected orientable surface of genus at least two. Its Teichm\"uller space $\Tc(S)$ is the space of marked hyperbolic structures on $S$.
A marked hyperbolic structure on $S$ (up to isometry) is equivalent to a discrete faithful representation $\rho:\pi_1(S)\to \mathsf{PO}(2,1)$ (up to conjugacy). Given $[\rho]\in\Tc(S)$, 
one obtains a hyperbolic surface $X_\rho:=\mathbb H^2/\rho(\pi_1(S))$ and a homeomorphism $h_\rho:S\to X_\rho$  (well-defined up to homotopy) in the homotopy class of $\rho$, 
called a marking.
One may also regard $\Tc(S)$ as a connected component of the character variety 
$$\mathcal X(\pi_1(S), \mathsf{PO}(2,1)) := \mathrm{Hom}(\pi_1(S),\mathsf{PO}(2,1))/\mathsf{PO}(2,1).$$
For any $g\in \pi_1(S)$, let $\ell_{\rho}(g)$ be the translation length  of $\rho(g)$ on $\Hb^2$, which is also the length of the geodesic representative of the curve  on $X_\rho$ in
the free homotopy class determined by $g$.

In \cite{thurston1998minimal}, Thurston defined an asymmetric metric on $\Tc(S)$, now called Thurston's asymmetric metric, 
by considering \emph{maximal stretching} between two marked hyperbolic surfaces. More precisely, for $[\rho],[\sigma]\in \Tc(S)$, 
 Thurston defined the maximal stretching constant
\begin{equation} \label{eqn:max stretching}
L_{\rho\sigma} := \sup_{g \in \pi_1(S) - \{1\}} \frac{\ell_\sigma(g)}{\ell_\rho(g)}
\end{equation}
and showed $1 \le L_{\rho\sigma} < + \infty$, with the equality if and only if $[\rho]=[\sigma]$. Then
$$
\dist_{\rm Thurston}([\rho],[\sigma]) := \log L_{\rho\sigma}
$$
is Thurston's asymmetric metric on $\Tc(S)$. Thurston also showed that $L_{\rho\sigma}$ is the minimal Lipschitz constant of any homeomorphism from $X_\rho$ to $X_\sigma$
in the homotopy class of $h_\sigma\circ h_\rho^{-1}$,  which is the homotopy class respecting the markings. 
Notice that 
$$
\inf_{g \in \pi_1(S) - \{1\}} \frac{\ell_\sigma(g)}{\ell_\rho(g)} = \frac{1}{L_{\sigma\rho}}.
$$

Many authors, see Section \ref{history}, have studied the exponential growth rate
of the number of free homotopy classes of closed curves in $S$ whose length ratio $\frac{\ell_{\sigma}(\cdot)}{\ell_{\rho}(\cdot)}$ is close to  a fixed ratio $R$. 
More precisely, for $R \in \left[ \frac{1}{L_{\sigma \rho}}, L_{\rho \sigma} \right]$, we call the exponential growth rate
\begin{equation} \label{eqn:ratio entropy}
\Ec_{\rho \sigma}(R) := \lim_{\epsilon \to 0} \limsup_{T \to + \infty} \frac{\log \# \left\{ [g] \in [\pi_1(S)] : \abs{\frac{\ell_\sigma(g)}{\ell_\rho(g)}- R} < \epsilon, \  \ell_\rho(g) \le T \right\}}{T}
\end{equation}
the \emph{entropy along the ratio $R$}, where $[\pi_1(S)]$ is the collection of non-trivial conjugacy classes in $\pi_1(S)$. 
In particular, we call each of the entropies along the two extreme ratios
$$
\Ec_{\rho \sigma}(L_{\rho \sigma}) \quad \text{and} \quad \Ec_{\rho \sigma}(1/L_{\sigma \rho})
$$
\emph{extremal entropies}. We will later consider  notions of entropy and extremal entropy in a more general setting.

If $R$ is not extreme, i.e. $
R \in \left( \frac{1}{L_{\sigma\rho}}, \ L_{\rho\sigma} \right)
$, 
then it follows from works of Sambarino \cite[Corollary 4.4]{Sambarino_hyperconvex} and Chow--Fromm \cite[Corollary 7.8]{CF_joint} 
that $\Ec_{\rho \sigma}(R) > 0$, and that $\Ec_{\rho \sigma}(R)$ is given by 
the so-called growth indicators for higher-rank discrete subgroups introduced by Quint \cite{Quint_divergence}. We will discuss this more fully
in Section \ref{subsec:GI} below. The special case where $R = 1$ was studied earlier by Schwartz--Sharp~\cite{SS_correlation}.

Non-extremal entropies above were 
studied using  techniques based on higher-rank Patterson--Sullivan theory
\cite{Quint_PS,Sambarino_hyperconvex,Sambarino_report}.
 On the other hand, the higher-rank Patterson--Sullivan theory cannot be applied to the 
counting problem for extremal entopries, due to the lack of an appropriate 
Patterson--Sullivan measure. 

In this paper, we show that the 
growth rate of closed geodesics which are almost maximally (or minimally) stretched is subexponential for two marked closed hyperbolic surfaces, and 
hence that extremal entropies are always zero in this case. This completes the above counting problem by handling the  extreme cases.

\begin{theorem}[Zero extremal entropies] \label{thm:subexp}
If $[\rho]\ne[\sigma]\in \Tc(S)$, then both extremal entropies are zero:
$$
\Ec_{\rho \sigma}(L_{\rho \sigma}) = \Ec_{\rho \sigma}(1/L_{\sigma \rho}) = 0.
$$
In other words,
\begin{enumerate}
    \item almost maximally stretched closed geodesics have subexponential growth:
    $$
\lim_{\epsilon \to 0} \limsup_{T \to + \infty} \frac{\log \# \left\{ [g] \in [\pi_1(S)] : \frac{\ell_\sigma(g)}{\ell_\rho(g)} \ge L_{\rho\sigma} - \epsilon \text{ and } \ell_\rho(g) \le T \right\}}{T} = 0.
$$
\item almost minimally stretched closed geodesics have subexponential growth:
    $$
\lim_{\epsilon \to 0} \limsup_{T \to + \infty} \frac{\log \# \left\{ [g] \in [\pi_1(S)] : \frac{\ell_\sigma(g)}{\ell_\rho(g)} \le \frac{1}{L_{\sigma\rho}} + \epsilon \text{ and } \ell_\sigma(g) \le T \right\}}{T} = 0.
$$
\end{enumerate}
\end{theorem}

Note that in Theorem \ref{thm:subexp}, Item (2) follows from Item (1) by switching the roles of $\rho$ and $\sigma$.

\subsection{Comments on the proof}
The proof of Theorem \ref{thm:subexp} makes use of the strengthening of Thurston's results by Gu\'eritaud--Kassel \cite{GK_maximally} and 
a generalization of the Variational principle due to Kifer \cite{Kifer_large} (see also \cite{Pollicott_large}).

We hope that this paper  will be the  starting point  for a more general study of extremal entropy and  that our argument 
is flexible enough to be extended to a variety of other settings where analogues of Thurston's results hold.
We also expect that arguments presented in this paper can also be applied to several types of counting problems.

We finally note that while we present our argument in a very general setting, one can also translate it into  the language of geodesic currents on surfaces, 
if one is only concerned with closed hyperbolic surfaces. We give more details in  Remark \ref{currents proof}.

\subsection{A higher-rank viewpoint on Theorem \ref{thm:subexp}} \label{subsec:GI}
In the spirit of early works of Bishop--Steger \cite{bishop-steger} and Burger \cite{Burger_Manhattan}, 
we may re-interpret this result in terms of the product representation into 
$\mathsf{PO}(2,1)\times\mathsf{PO}(2,1)$, which is a rank-two semisimple Lie group.
If $[\rho], [\sigma] \in \Tc(S)$ are  distinct marked hyperbolic surfaces, then
the image
$$
\Gamma := (\rho\times\sigma) \left(\pi_1(S) \right) := \left\{ (\rho(g), \sigma(g)) : g \in \pi_1(S) \right\} \subset \mathsf{PO}(2,1)\times\mathsf{PO}(2,1)
$$
is a Zariski dense discrete subgroup.  In the language of Anosov representations,  $\Gamma$ is  a Borel Anosov subgroup of $\mathsf{PO}(2,1)\times\mathsf{PO}(2,1)$.
Then one considers the closed cone, 
called the \emph{Benoist limit cone},
\begin{equation} \label{eqn:limit cone}
\Bc_{\Ga} := \overline{\Rb_{> 0} \cdot \{ (\ell_{\rho}(g), \ell_{\sigma}(g))  \in \Rb^2 : g \in \pi_1(S) \}} \subset \Rb^{2}.
\end{equation}
which may naturally be identified with a closed cone in the positive Weyl chamber of 
$\mathsf{PO}(2,1)\times\mathsf{PO}(2,1)$. 
It was introduced by Benoist \cite{Benoist_properties} and plays a significant role in the study of 
discrete subgroups of Lie groups. Benoist further showed that if $\Gamma$ is Zariski dense, then $\Bc_{\Ga}$ is convex, and hence in our setting, it can be characterized as
\begin{equation} \label{eqn:character Benoist}
\Bc_{\Ga} = \left\{ (u_1, u_2) \in \Rb_{> 0}^2 : \frac{1}{L_{\sigma\rho}} \le \frac{u_2}{u_1} \le L_{\rho\sigma} \right\} \cup \{0\}.
\end{equation}
We also note that, by Thurston \cite{thurston1998minimal}, $\Bc_{\Gamma}$ has non-empty interior in this case, and this is extended to general Zariski dense discrete subgroups of  semisimple Lie groups by 
Benoist.

In order to study asymptotic properties of discrete subgroups, Quint \cite{Quint_divergence} introduced the notion of growth indicator. 
The \emph{growth indicator} of $\Ga$ in our setting is the function
\begin{equation} \label{eqn:GI}
  \psi_{\Ga} : \Rb^2 \to [0, +\infty) \cup \{-\infty\}
\end{equation}
defined as follows: fix a basepoint $o \in \Hb^{2}$ 
and consider the vector
$$\kappa(g) := \big(\dist_{\Hb^{2}}(o, \rho(g) o), \dist_{\Hb^{2}}(o, \sigma(g)( o))\big) \in \Rb^2$$ 
for each $g\in \pi_1(S)$, which represents displacements on each component. (This 
agrees with the Cartan projection with respect to an appropriate choice of Cartan decomposition of $\PO(2, 1) \times \PO(2, 1)$.)
Then for $u \in \Rb^2$,
$$
\psi_{\Ga}(u) := \norm{u} \cdot \inf_{\Cc \ni u} \limsup_{T \to + \infty} \frac{\log \# \{ g \in \pi_1(S) : \kappa(g) \in \Cc, \ T - 1 \le \norm{\kappa(g)} \le T\}}{T},
$$
where the infimum is over all open cone $\Cc \subset \Rb^2$ containing $u$ and with the convention $0 \cdot (\pm \infty) = 0$. 
This definition does not depend on the choice of $o \in \Hb^2$ and of the norm $\norm{ \cdot}$ on $\Rb^2$.
One can also see that the growth indicator is positively homogeneous of degree one.  

The growth indicator is a higher-rank generalization of the critical exponent in rank one, since 
one may think of $\psi_{\Ga}(u)$ as the exponential growth rate  of $\Gamma$ in the direction $u$. 
Quint \cite{Quint_divergence} proved that
$$
\Bc_{\Ga} = \{ u \in \Rb^2 : \psi_{\Ga}(u) \ge 0\} \quad \text{and} \quad 
\psi_{\Ga}(u) > 0 \quad \text{for all } u \in \Bc_{\Ga}^{\circ}
$$
where $\Bc_{\Ga}^{\circ}$ is the interior of $\Bc_{\Ga}$. 
 Moreover, he showed that $\psi_{\Ga}(u) = -\infty$ for $u \notin \Bc_{\Gamma}$ and  $\psi_{\Ga}$ is concave and upper semicontinuous
on $\Bc_\Gamma$.

However, very little is known about the  value of the growth indicator on the boundary $\partial \Bc_{\Ga}$.
In this language, Theorem \ref{thm:subexp} implies that the growth indicator in fact vanishes on the boundary when $[\rho]\ne[\sigma]\in\Tc(S)$.

\begin{corollary}[Vanishing of the growth indicator on the boundary] \label{cor:GI}
If $[\rho]\ne[\sigma]\in\Tc(S)$ and 
$\Ga := (\rho \times \sigma)\big(\pi_1(S)\big) \subset \mathsf{PO}(2,1)\times\mathsf{PO}(2,1)$, then 
    $$
    \psi_{\Ga} = 0 \quad \text{on } \partial \Bc_{\Ga}.
    $$
\end{corollary}

In \cite{chow2023jordan}, Chow--Oh proved that the growth indicators defined in terms of the Cartan and Jordan projections coincide on the interior $\Bc_{\Gamma}^{\circ}$, 
for a general Zariski dense Borel Anosov subgroup of a semisimple Lie group. Our Theorem \ref{thm:subexp} and Corollary \ref{cor:GI} imply that, in our setting, 
they coincide even on the boundary $\partial \Bc_{\Ga}$, and hence on the whole space $\Rb^2$.

\subsection{General Borel Anosov subgroups of $\PO(2, 1) \times \PO(2, 1)$}

Indeed, any 
 non-elementary 
Borel Anosov subgroup $\Gamma \subset \PO(2, 1) \times \PO(2, 1)$ is virtually the image of product representation $(\rho \times \sigma)(G)$, where $G$ is 
either a closed surface group or a free group, and $\rho, \sigma : G \to \PO(2, 1)$ are non-elementary convex cocompact representations.  
If $\rho$ and $\sigma$ are not conjugate, then $\Gamma$ is Zariski dense.
Hence, the Benoist limit cone $\Bc_{\Gamma}$ and the growth indicator $\psi_{\Gamma}$ of $\Gamma$
 can also be defined using
 Equations \eqref{eqn:limit cone} 
and \eqref{eqn:GI} in this more general setting.

As a consequence of our study of extremal entropy, we also observe that the growth indicator must vanish on at least one boundary component of the Benoist limit cone. We denote by 
$$\partial_{\max}, \ \partial_{\min} \subset \partial \Bc_{\Gamma} - \{0\}$$ 
the components of $\partial \Bc_{\Gamma} - \{0\} \subset \Rb^2$ with maximal and minimal slopes, respectively. See Figure \ref{fig:free examples} below. Notice that as in Equation \eqref{eqn:character Benoist}, the slopes of $\partial_{\max}$ and $\partial_{\min}$ are $L_{\rho \sigma}$ and $1 / L_{\sigma \rho}$, respectively.

\begin{theorem}[Necessary vanishing] \label{thm:necessary vanishing}
Let  $\Gamma \subset \PO(2, 1) \times \PO(2, 1)$ be a Zariski dense
Borel Anosov subgroup. Then $\psi_{\Gamma} = 0$ on at least one component of $\partial \Bc_{\Gamma}$. 

More precisely:
\begin{enumerate}
  \item If the slope of $\partial_{\max}$ is strictly bigger than $1$, then 
  $$
  \psi_{\Gamma}= 0 \quad \text{on } \partial_{\max}.
  $$
  \item If the slope of $\partial_{\min}$ is strictly smaller than $1$, then 
  $$
  \psi_{\Gamma}= 0 \quad \text{on } \partial_{\min}.
  $$
\end{enumerate}
\end{theorem}
See Corollary \ref{cor:necessary vanishing general} for a more general version of this statement.

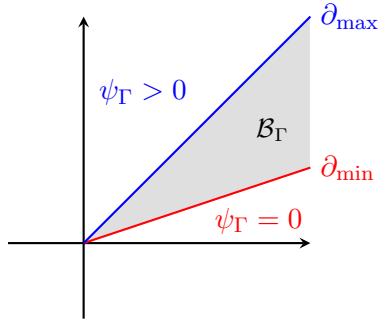
\begin{figure}[ht]
\begin{tikzpicture}[>= stealth, scale=1]

\filldraw[draw= white, fill = gray!50, opacity=0.5] (3, 3) -- (0, 0) -- (3, 1) -- (3, 3) -- (3, 3);

\draw[->, thick] (-1, 0) -- (3, 0);
\draw[->, thick] (0, -1) -- (0, 3);

\draw[red, thick] (0, 0) -- (3, 1);

\draw[red] (3, 1) node[right] {$\partial_{\min}$};

\draw[red] (1.6, 0.3) node[right] {$\psi_{\Gamma} = 0$};

\draw[blue, thick] (0, 0) -- (3, 3);

\draw[blue] (3, 3) node[right] {$\partial_{\max}$};

\draw[blue] (1.5, 2) node[left] {$\psi_{\Gamma} > 0$};

\draw (2.5, 1.5) node {\small $\Bc_{\Gamma}$};

\end{tikzpicture}

\caption{Theorem \ref{thm:necessary vanishing}(2) and Proposition \ref{free examples}} \label{fig:free examples}
\end{figure}

We also observe that the strict inequality between slopes is necessary to obtain the vanishing in Theorem \ref{thm:necessary vanishing}. Indeed, 
given any real number in $(0, 1)$, we construct a product of convex cocompact representations 
whose extremal entropy realizes the prescribed number. 
Let $\Sigma_{g, b}$ denote the compact connected orientable surface of genus $g$ with $b$ boundary components.
We say that  a convex cocompact representation $\rho: \pi_1(\Sigma_{g, b}) \to \PO(2, 1)$ uniformizes $\Sigma_{g,b}$ if there exists a homeomorphism
$h:\Sigma_{g,b}^\circ\to X_\rho$ in the homotopy class determined by $\rho$, where $\Sigma_{g,b}^\circ$ is the interior of $\Sigma_{g,b}$.

\begin{proposition}[Arbitrary positive values]
\label{free examples} 
Suppose $2g + b > 3$  and \hbox{$b>0$}. Then for any $\delta \in (0, 1)$, there exist non-conjugate convex cocompact representations $\rho, \sigma : \pi_1(\Sigma_{g, b}) \to \PO(2, 1)$ 
which uniformize $\Sigma_{g,b}$ so that: 
\begin{enumerate}
  \item $L_{\rho \sigma} = 1$.
  \item $\Ec_{\rho \sigma}(1) = \delta$.
  \item If $\Ga := (\rho \times \sigma)(\pi_1(\Sigma_{g, b}))$, then $\partial_{\max} = \{(t, t) : t > 0 \}$ has slope 1 and  
  $$\psi_{\Gamma}(t, t) = \delta \cdot t \quad \text{for } t > 0.$$
\end{enumerate}
\end{proposition}

Notice that in this situation $\Gamma$ is Zariski dense in $\mathsf{PO}(2,1)\times\mathsf{PO}(2,1)$. See Figure \ref{fig:free examples} above
 for a picture of Proposition \ref{free examples}.

\begin{remark}
Recall that  if $b>0$, then $\pi_1(\Sigma_{g,b})$ is a free group of rank $2g+b-1$, so  every finitely generated free group of rank at least 3 occurs as
$\pi_1(\Sigma_{g,b})$ for some pair $(g,b)$ with $2g+b>3$ and $b>0$. The only surfaces, with $b>0$, whose interiors admit complete hyperbolic structures 
which are omitted are the one-holed torus and the three-holed sphere.

\end{remark}

\begin{remark} \label{rmk:tube}
In \cite{chow2023jordan}, Chow--Oh also considered entropy defined in terms of the exponential growth rate within tubes, instead of cones as in Equation~\eqref{eqn:ratio entropy}. 
 Items (2) and~(3) in Proposition \ref{free examples}, together with the comparison 
 in \cite[Theorem 8.1]{chow2023jordan}, imply that this tube-type entropy can also be 
 made an arbitrary number in $(0, 1)$.
\end{remark}

\subsection{Products of convex cocompact representations into $\mathsf{PO}(3,1)$}

When we consider representations into  $\mathsf{PO}(3,1)$  the situation is
more complicated even for convex cocompact representations of closed surface groups, a.k.a. quasifuchsian representations. 
For a closed surface $S$ of genus at least two, we consider quasifuchsian representations $\rho, \sigma : \pi_1(S) \to \PO(3, 1)$ and the image $\Gamma := (\rho \times \sigma)(\pi_1(S)) \subset \PO(3, 1) \times \PO(3, 1)$ of their product representation. All the notions described above can naturally be generalized to this setting, and we will  use the same notation.

\begin{proposition} \label{QF examples}
Let $S$ be a closed surface of genus at least two.
\begin{enumerate}
\item 
There exists a Fuchsian representation $\rho_1:\pi_1(S)\to\mathsf{PO}(2,1)\subset \mathsf{PO}(3,1)$
and a quasifuchsian, but not Fuchsian,
representation $\sigma_1: \pi_1(S)\to \mathsf{PO}(3,1)$, so that for $\Gamma_1 := (\rho_1 \times \sigma_1)(\pi_1(S))$,
$$
\partial_{\max} \text{ has slope } 1 \quad \text{and} \quad \psi_{\Gamma_1} > 0 \text{ on } \partial_{\max}.
$$

\item

There exist non-conjugate quasifuchsian, but not Fuchsian, representations $\rho_2, \sigma_2 : \pi_1(S) \to \PO(3, 1)$, so that for $\Gamma_2 := (\rho_2 \times \sigma_2)(\pi_1(S))$,
$$
\partial_{\max} \text{ has slope } 1 \quad \text{and} \quad \psi_{\Gamma_2} > 0 \text{ on } \partial_{\max} .
$$
\end{enumerate}
\end{proposition}

Notice that $\Gamma_1$ is Zariski dense in $\mathsf{PO}(2,1)\times\mathsf{PO}(3,1)$, while
$\Gamma_2$ is Zariski dense in $\mathsf{PO}(3, 1)\times\mathsf{PO}(3,1)$.
We also refer to Gu\'eritaud--Kassel \cite[Section 9.4, Section 10]{GK_maximally} for several illustrative examples. 
In particular, their Example~10.5 describes convex cocompact representations $\rho$ and $\sigma$  of  a rank-two free group, 
where $L_{\rho\sigma}<1$ and $L_{\rho\sigma}$ is not equal  to
the minimal Lipschitz constant of a map $X_\rho \to X_\sigma$ (in the appropriate homotopy class).

\subsection{Historical remarks}
\label{history}
Counting the number of closed geodesics on a closed hyperbolic surface in terms of their lengths is first due to Huber \cite{Huber_counting}, based on Selberg's trace formula. Such a counting result is also called the prime geodesic theorem. Margulis \cite{Margulis_counting,Margulis_monograph} developed a dynamical framework of mixing and equidistribution, and established the counting for compact negatively curved manifolds. For convex cocompact hyperbolic manifolds, the counting was proved by Lalley \cite{Lalley_counting}, using renewal theory.
In a more general setting of $\CAT(-1)$ spaces and their discrete isometry groups, Roblin \cite{Roblin2003ergodicite} extended Margulis' approach based on the theory of conformal measures developed by Patterson \cite{Patterson_fuchsian} and Sullivan \cite{Sullivan_density}.

The study of directional asymptotic growth of products of Fuchsian representations was initiated by Bishop--Steger \cite{bishop-steger} 
which was further generalized by Burger \cite{Burger_Manhattan}.
Schwartz--Sharp \cite{SS_correlation} then counted the number of closed geodesics whose lengths are similar on the given two marked hyperbolic surfaces, based on Lalley's work \cite{Lalley_distribution}.
In other words, their work is related to the entropy along the ratio $1$ which is defined in 
Equation \eqref{eqn:ratio entropy}. As in Section \ref{subsec:GI} this can also be interpreted as a counting result for a discrete subgroup of $\PO(2, 1) \times \PO(2, 1)$, which is a higher-rank Lie group.
For general Anosov subgroups in higher rank, the counting problem corresponding to non-extremal entropy was established by Sambarino 
\cite{Sambarino_quantitative,Sambarino_hyperconvex} and Chow--Fromm \cite{CF_joint}. It was further extended by Bray--Canary--Kao--Martone \cite{BCKM_counting} to relatively Anosov representations of surface groups, and by Blayac--Canary--Zhu--Zimmer \cite{BCZZ_counting} to general Gromov--Patterson--Sullivan systems which include all relatively Anosov subgroups. The works 
of Sambarino and Bray--Canary--Kao--Martone are based on thermodynamic formalism, while the work of Blayac--Canary--Zhu--Zimmer generalizes Margulis' framework.

Asymptotic counting results for cones about internal vectors in the Benoist limit cone were also  established by Chow--Fromm \cite{CF_joint} and 
Chow--Oh \cite{chow2023jordan,CO_multiple} for products of   Anosov representions. Notice that the
 counting involved in Equation \eqref{eqn:ratio entropy} can be regarded as an asymptotic counting result for cones.
Chow--Oh  also established counting results for many different shapes of neighborhoods of vectors internal to the Benoist limit cone, such as tubes and finite boxes. 
The original work of Schwartz--Sharp \cite{SS_correlation} also takes the form of counting in boxes. The counting results of Chow--Fromm and Chow--Oh are based on the local mixing of the one-dimensional diagonal flow, which was established by Sambarino \cite{Sambarino_report} and Chow--Sarkar \cite{CS_local}.

We finally remark that all higher-rank results mentioned above do not handle extremal entropies. The main novelty of this paper is to determine extremal entropies, 
sharpening Thurston's result \cite{thurston1998minimal} on the uniqueness of the geodesic lamination containing all measured laminations minimizing the Lipschitz constant 
between two closed marked hyperbolic surfaces.

\subsection*{Acknowledgements}
The authors would like to thank Jeff Danciger, Fran\c{c}ois Gu\'eritaud and Fanny Kassel for helpful conversations. The authors also thank Hee Oh for useful comments 
on an earlier version of our manuscript and for pointing out Remark \ref{rmk:tube}.

\section{Entropy and geodesic flow} \label{sec:entropy}

In this section, we record basic facts about entropy which will be used in our work. We refer readers to \cite{EL_diagonal} for a more comprehensive exposition.

\subsection{Metric entropy}

Let $(\Xc, \mu)$ be a probability space.
For a (measurable) finite partition $\Pc$ of $\Xc$, its static entropy is defined by
$$
H_{\mu}(\Pc) := - \sum_{P \in \Pc} \mu(P) \log \mu(P)
$$
with the convention that $0 \cdot (- \infty) = 0$. For partitions $\Pc_1, \dots, \Pc_n$ of $\Xc$, we set 
$$
\bigvee_{i = 1}^n \Pc_i := \left\{ \bigcap_{i = 1}^{n} P_i : P_i \in \Pc, \ \forall i = 1, \dots, n \right\}.
$$

We now define the metric entropy of a probability-measure-preserving dynamical system. 

\begin{definition}
    The \emph{metric entropy} of a $\mu$-measure-preserving transformation $\varphi : \Xc \to \Xc$ with respect to $\mu$ is 
    $$
    h_{\mu}(\varphi) := \sup_{\Pc} \lim_{N \to + \infty} \frac{1}{N} \cdot H_{\mu} \left( \bigvee_{i = 0}^{N-1} \varphi^{-i} \Pc \right)
    $$
    where the supremum is over all finite partitions $\Pc$ of $\Xc$.
\end{definition}

For a $\mu$-measure-preserving flow $\{\varphi_t\}_{t \in \Rb}$ on $\mathcal{X}$, we have $h_\mu(\varphi_t) = |t| \cdot~h_\mu(\varphi_1)$ for all $t \neq 0$. Hence, the \emph{metric entropy of the flow} $\{\varphi_t\}_{t \in \Rb}$ with respect to $\mu$ is defined as $$h_\mu(\{\varphi_t\}_{t \in \Rb}) := h_\mu(\varphi_1) .$$

\subsection{Topological entropy and Variational principle}

We now define the notion of entropy for a topological dynamical system. 
Suppose that $(\Xc, \dist_{\Xc})$ is a compact metric space and $\varphi : \Xc \to \Xc$ is continuous.  For  $k \in \Nb$ and $\epsilon > 0$, two points $x, y \in \Xc$ are said to be $(k, \epsilon)$-separated if for some $0 \le i < k$, we have $\dist_{\Xc}(\varphi^i x, \varphi^i x') \ge \epsilon$. We denote by $N( \varphi, k, \epsilon)$ the maximal cardinality of a $(k, \epsilon)$-separated subset of $\Xc$, i.e. any two points in the subset are $(k, \epsilon)$-separated.

\begin{definition}
    The \emph{topological entropy} of a continuous map $\varphi : \Xc \to \Xc$ is
    $$
    h_{\rm top}(\varphi) := \lim_{\epsilon \to 0} \limsup_{k \to + \infty} \frac{1}{k} \cdot \log N(\varphi, k, \epsilon)
    $$
For a continuous flow $\{\varphi_t\}_{t \in \Rb}$ on $\mathcal{X}$, we define the  \emph{topological entropy of the flow} $\{\varphi_t\}_{t \in \Rb}$ by 
$$h_{\rm top}(\{\varphi_t\}_{t \in \Rb}) := h_{\rm top} (\varphi_1). $$
\end{definition}

An important property of topological entropy is the following Variational principle, which relates the topological entropy with metric entropy.

\begin{theorem}[Variational principle \cite{Goodman_entropy}] \label{thm:variational classic}
    Let $\Xc$ be a compact metric space and let $\varphi : \Xc \to \Xc$ be a homeomorphism. Then
    $$
    h_{\rm top}(\varphi) = \sup_{\mu} h_{\mu}(\varphi)
    $$ 
    where the supremum is over all $\varphi$-invariant probability measures on $\Xc$. The same also holds for flows of homeomorphisms on $\Xc$.
\end{theorem}

\subsection{The geodesic flow on the non-wandering set}
Suppose that $G$ is a non-elementary, torsion-free, 
Gromov hyperbolic group and \hbox{$\rho:G\to\mathsf{PO}(d,1)$} is a convex cocompact representation.  Recall that $\rho$ is convex cocompact
if and only if whenever $x_0\in\mathbb H^d$, the orbit map  $\tau:G\to\mathbb H^d$ given by
$\tau(g)=g(x_0)$ is a quasi-isometric embedding.
 We then set $X_\rho :=\mathbb H^d/\rho(G)$. 

The Hopf parametrization identifies the unit tangent bundle 
$T^1\mathbb H^d$ with $\big(\partial\mathbb H^d\times~\partial \mathbb H^d-\Delta\big)\times\mathbb R$  (where $\Delta$ is the diagonal) 
in such a way that the geodesic flow on $T^1\mathbb H^d$
is given by $\varphi_t(x,y,s)=(x,y,s+t)$ for all $(x,y,s)\in T^1\mathbb H^d$ and $t\in\mathbb R$. The geodesic flow on $T^1\mathbb H^d$ descends
to the geodesic flow on $T^1X_\rho$.

However, since we are interested in counting closed geodesics, it is natural to restrict to the non-wandering portion $T^1X_\rho^{(nw)}$ which is the closure
of the set of closed periods in $T^1X_\rho$. In terms of the Hopf parametrization, if $\Lambda(\rho)\subset \partial\mathbb H^d$ is the limit set of $\rho(G)$, then $T^1X_\rho^{(nw)}$
is the quotient of $\big(\Lambda(\rho)\times\Lambda(\rho)-\Delta)\times\mathbb R$ by $\rho(G)$. Since $\rho$ is convex cocompact, $T^1X_\rho^{(nw)}$
is compact.  This compactness allows us to apply standard techniques from dynamics. 

Throughout the paper, we consider only the entropy of the geodesic flow but vary the marked hyperbolic structure.  In this regard,  we omit the flow when writing entropy and denote by $h_{\rm top}(\rho)$ the topological entropy of the geodesic flow on $T^1X_\rho^{(nw)}$.
Similarly, for a probability measure $\mu$ on $T^1X_\rho^{(nw)}$ invariant under the geodesic flow, we let
$h(\mu)$ denote the metric entropy of the geodesic flow  on $T^1X_\rho^{(nw)}$ with respect to $\mu$.

Let 
$$
\Mc(\rho) := \left\{ 
  \begin{matrix}
    \text{ probability measures on } T^1X_\rho^{(nw)} \\
    \text{ invariant under the geodesic flow } 
  \end{matrix}
  \right\},
$$
which is a compact metrizable space when equipped with the weak-* topology.
The Variational principle (Theorem \ref{thm:variational classic}) is restated as 
$$
h_{\rm top}(\rho) = \sup_{\mu \in \Mc(\rho)} h(\mu).
$$

The topological entropy turns out to be the same as the exponential growth rate of the number of closed geodesics.  Indeed, denoting by $\ell_{\rho}(g)$ the translation length of $\rho(g)$ on $\Hb^d$ for  $g \in G$, Huber \cite{Huber_counting}, Margulis \cite{Margulis_counting, Margulis_monograph}, Patterson \cite{Patterson_counting}, Lalley \cite{Lalley_counting},  Sullivan \cite{Sullivan_density,Sullivan_entropy}, and Roblin \cite{Roblin2003ergodicite} showed that 
\begin{equation} \label{eqn:Margulis}
  \begin{aligned}
h_{\rm top}(\rho) & = \lim_{T \to + \infty} \frac{\log \# \{ [g] \in [G] : \ell_\rho(g) \le T \}}{T} \\
& =\lim_{T \to + \infty} \frac{\log \# \{ g \in G : d_{\mathbb H^d} \big(o,\rho(g)(o)\big) \le T \}}{T}
  \end{aligned}
\end{equation}
for any $o \in \Hb^d$, where $[G]$ denotes the collection of non-trivial conjugacy classes in $G$, and $[g]$ is the conjugacy class of $g \in G$.
Moreover,\footnote{Here, $f(T) \sim g(T)$ means that the ratio $f(T)/g(T)$ converges to a positive and finite constant as $T \to + \infty$.} 
$$\# \{ [g] \in [G] : \ell_{\rho} (g) \le T \}\sim \frac{e^{h_{\rm top}(\rho)\cdot T}}{T},$$
so
\begin{equation}\label{eqn:Margulis2}
h_{\rm top}(\rho) = \limsup_{T \to + \infty} \frac{\log \# \{ [g] \in [G] : T-1 \le \ell_\rho(g) \le T \}}{T}.
\end{equation}

This Variational principle and Margulis' counting  results were generalized by Kifer \cite{Kifer_large} (see also \cite{Pollicott_large}). 
To state Kifer's theorem, we recall that each non-trivial conjugacy class $[g] \in [G]$ determines an oriented
closed geodesic on $X_\rho$ and hence a closed
orbit of the geodesic flow in $T^1 X_{\rho}$. One obtains
a measure $\widetilde \mu_{[g]}$ of  mass $\ell_\rho(g)$ supported on the closed orbit on $T^1 X_\rho^{(nw)}$ which is invariant under the geodesic flow. We then set 
\begin{equation} \label{eqn:periodic measure}
\mu_{[g]}:= \frac{1}{\ell_\rho(g)} \cdot \widetilde \mu_{[g]} \in \Mc(\rho).
\end{equation}
Note that $\mu_{[g]} \neq \mu_{[g^{-1}]}$. 
Let $[G]_{\rm prim}$  denote
the set of conjugacy classes of non-trivial primitive elements of $G$.
Sigmund \cite{Sigmund_hyperbolic} showed that measures supported on closed orbits form a dense subset of $\Mc(\rho)$, i.e.  
\begin{equation} \label{eqn:Sigmund}
    \overline{\{ \mu_{[g]} : [g] \in [G]_{\rm prim} \} } = \Mc(\rho).
\end{equation}

Then, in our setting,  Kifer's Variational principle takes the following form.

\begin{theorem}[{Kifer \cite{Kifer_large} (see also \cite{Pollicott_large})}] \label{thm:variational}
Suppose that $G$ is a non-elementary, torsion-free, Gromov hyperbolic group and 
$\rho:G\to\mathsf{PO}(d,1)$ is convex cocompact and 
$\Kc \subset \Mc(\rho)$ is a closed subset. Then 
$$
\limsup_{T \to + \infty} \frac{\log \left(\frac{\# \{ [g] \in [G]_{\rm prim} : \mu_{[g]} \in \Kc, \ T - 1 \le \ell_\rho(g) \le T\}}{\# \{ [g] \in [G]_{\rm prim} : T - 1 \le \ell_\rho(g) \le T \}} \right)}{T} \le \sup_{\mu \in \Kc} h(\mu) - h_{\rm top}(\rho).
$$
\end{theorem}

\subsection{Geodesic currents} \label{subsec:currents}

We will only use the theory of geodesic currents in the proof of Proposition \ref{free examples}. We will also explain, in Remark \ref{currents proof}, how one can use
geodesic currents to give an alternate proof of Theorem \ref{thm:subexp}.

Let $G$ be a non-elementary, torsion-free, Gromov hyperbolic group and set 
$$\partial^{(2)}G:=\partial G\times \partial G -  \Delta$$
where $\partial G$ is the Gromov boundary of $G$ and  $\Delta=\{(z,z):z\in \partial G\}$ is the diagonal.
A geodesic current on $G$ is a $G$-invariant Radon measure on
$\partial^{(2)}G$ which is invariant under the
involution $\iota:\partial^{(2)}G\to \partial^{(2)}G$
given by $\iota(x,y):=(y,x)$.
Let $\Cc(G)$ denote the space of geodesic currents on $G$, equipped with the weak-* topology. 

For each primitive non-trivial conjugacy class $[g] \in [G]_{\rm prim}$, 
one can naturally associate the geodesic current 
\begin{equation} \label{eqn:current for curve}
\nu_{[g]} \in \Cc(G)
\end{equation}
which is supported on the $G$-orbit of pairs of  fixed points of $g$ in $\partial G$
and gives each pair
weight $1/2$. 
For $[g^n] \in [G]$,  where $g\in G - \{1\}$ is primitive and $n\ge 2$, 
 we set $\nu_{[g^n]} := n \cdot \nu_{[g]}$. 
Bonahon \cite[Theorem 7]{bonahon-negative} showed that the set of  $\Rb_{>0}$-weighted multiples of currents of the form
$\nu_{[g]}$ is dense in $\Cc(G)$.

Let $\Sigma$ be a compact connected orientable surface of negative Euler characteristic, with (possibly empty) boundary.
We will use the shorthand $\Cc(\Sigma)$ for $\Cc(\pi_1(\Sigma))$.
Suppose that $\rho:\pi_1(\Sigma)\to \mathsf{PO}(2,1)$ is a convex cocompact representation 
which uniformizes $\Sigma$. 
Recall that 
there exists a unique $\rho$-equivariant homeomorphism $\xi_\rho:\partial \pi_1(\Sigma) \to \Lambda(\rho)\subset\partial \Hb^2$.
If $\mu\in\Mc(\rho)$ is a flow-invariant probability measure on $T^1X_\rho^{(nw)}$, then it lifts to a flow-invariant, $\rho(\pi_1(\Sigma))$-invariant measure $\widetilde\mu$ on 
$\left(\Lambda(\rho)\times\Lambda(\rho)-\Delta\right)\times\mathbb R$
which has the form $d\eta_\rho\otimes dt$  for some  $\rho(\pi_1(\Sigma))$-invariant Radon measure $\eta_\rho$ on $\Lambda(\rho)\times\Lambda(\rho)-\Delta$.
This  gives rise to a continuous map 
$$
\phi_\rho:\Mc(\rho) \to \Cc(\Sigma)\quad\text{given by}\quad \phi_{\rho}(\mu) : =\frac{\big(\xi_\rho\times \xi_\rho\big)^*\eta_\rho+\iota^*\Big( \big(\xi_\rho\times\xi_\rho\big)^*\eta_\rho\Big)}{2}.
$$
Notice that
$$ 
\phi_\rho(\mu_{[g]})=\frac{\nu_{[g]}}{\ell_\rho(g)}.$$

\section{Maximal stretching laminations  and extremal entropy}

We recall that Thurston \cite[Theorem 3.1, Theorem 8.5]{thurston1998minimal} showed  that for a closed surface $S$ of genus at least two, if $[\rho]\ne[\sigma]\in\Tc(S)$, then the maximal stretching constant $L_{\rho\sigma}$ 
is equal to the minimal Lipschitz constant  of a homeomorphism $X_\rho \to  X_\sigma$ in the homotopy class of $h_\sigma\circ h_\rho^{-1}$. Moreover,
$L_{\rho\sigma} > 1.$

We will use a strengthening and generalization of Thurston's work due to  Gu\'eritaud--Kassel \cite{GK_maximally}. The {\em  stretch locus} of an $L$-Lipschitz map $f:A\to B$
is the set of points $a\in A$, so that the restriction of $f$ to any open neighborhood of $a$ is not $L'$-Lipschitz for any $L' < L$. A {\em geodesic lamination} on a hyperbolic manifold $X$
is a closed subset  of $X$ which is a disjoint union of complete simple geodesics.

\begin{theorem}{\rm (Gu\'eritaud--Kassel \cite[Theorem 1.3, Corollary 1.12, Lemma 4.7, Lemma 4.10]{GK_maximally})}
\label{MaximalGK}
Suppose that $G$ is a non-elementary, torsion-free, Gromov hyperbolic group and
$\rho,\sigma:G\to \mathsf{PO}(d,1)$  are 
convex cocompact representations. 
If $L_{\rho\sigma}>1$, then there exists an 
$L_{\rho\sigma}$-Lipschitz map $f_{\rho\sigma}:X_\rho\to X_\sigma$, 
in the homotopy class determined by $\sigma \circ \rho^{-1}$,
whose  stretch locus  is a geodesic lamination $\lambda_{\rho\sigma}$ contained in the convex core of $X_\rho$. 
\end{theorem}

We use this to show the following upper bound on an extremal entropy which will nearly immediately imply our main theorem.
If $\lambda$ is a geodesic lamination on a convex cocompact  hyperbolic manifold $X_\rho$, let 
$$\Mc(\lambda):=\left\{\mu\in\Mc(\rho) : \pi(\supp \mu)\subset \lambda\right\}$$
where $\pi:T^1X_\rho\to X_\rho$ is the basepoint projection map. Recall $\Ec_{\rho \sigma}(\cdot)$ from Equation \eqref{eqn:ratio entropy}.

\begin{theorem}
\label{general entropy}
Suppose that $G$ is a non-elementary, torsion-free, Gromov hyperbolic group and
 $\rho,\sigma:G\to\mathsf{PO}(d,1)$  are 
convex cocompact representations. 
If $L_{\rho\sigma}>1$, then 
$$
\Ec_{\rho \sigma}(L_{\rho \sigma}) \le \sup_{\mu \in \Mc(\lambda_{\rho\sigma})} h(\mu)
$$
 where $\lambda_{\rho\sigma}$ is a geodesic lamination contained in the convex core of $X_{\rho}$ which is the stretch locus for an $L_{\rho\sigma}$-Lipschitz map
 $f_{\rho\sigma}:X_\rho\to X_\sigma$ in the homotopy class determined by $\sigma \circ \rho^{-1}$.
 \end{theorem}
 
 \begin{proof}  
Recall that
  $$
\Ec_{\rho \sigma}(L_{\rho \sigma}) = \lim_{\epsilon\to 0}\limsup_{T \to + \infty}  \frac{\log \# \left\{ [g] \in [G] :  \frac{\ell_\sigma(g)}{\ell_\rho(g)} \ge L_{\rho \sigma} - \epsilon,  \ell_\rho(g) \le T \right\}}{T} .
  $$
 First note that it suffices to prove the desired estimate within $[G]_{\rm prim}$, as it does not affect the exponential growth rate.
 Indeed, for each $\epsilon, T > 0$, we have
 $$ \begin{aligned}
   & \# \left\{ [g] \in [G] :  \frac{\ell_\sigma(g)}{\ell_\rho(g)} \ge L_{\rho \sigma} - \epsilon,  \ell_\rho(g) \le T \right\}  \\
   & \quad \le \frac{T}{\operatorname{sys}(X_\rho)} \cdot \# \left\{ [g] \in [G]_{\rm prim} :  \frac{\ell_\sigma(g)}{\ell_\rho(g)} \ge L_{\rho \sigma} - \epsilon,  \ell_\rho(g) \le T \right\}
 \end{aligned}$$
 where $\operatorname{sys}(X_\rho) > 0$ is the length of the shortest closed geodesic in $X_\rho$. Since $[G]_{\rm prim} \subset [G]$, the above inequality implies that exponential growth rates we consider are the same for $[G]$ and $[G]_{\rm prim}$.

Let $f_{\rho \sigma}$ and $\lambda_{\rho \sigma}$ as in Theorem \ref{MaximalGK}.
For simplicity, we write $L:=L_{\rho\sigma}$, $\lambda:=\lambda_{\rho\sigma}$, and $f:=f_{\rho\sigma}$ for the remainder of the proof.
 
For $n\in \mathbb N$, let
\begin{equation} \label{eqn:closedset}
\Kc_n:=\overline{ \left\{ \mu_{[g]} \in \Mc(\rho) :[g]\in [G]_{\rm prim}\text{ and } \frac{\ell_\sigma(g)}{\ell_\rho(g)} \ge L - \frac{1}{n} \right\} }.
\end{equation}
Applying Theorem \ref{thm:variational} and Equation \eqref{eqn:Margulis2}, we conclude that
$$\begin{aligned}
\limsup_{T \to + \infty} & \frac{\log \# \left\{ [g] \in [G]_{\rm prim} :  \frac{\ell_\sigma(g)}{\ell_\rho (g)} \ge L - \frac{1}{n}, \ T - 1 \le \ell_\rho(g) \le T \right\}}{T} 
 \le \sup_{\mu \in \Kc_n} h(\mu).
\end{aligned}
$$

Suppose $\mu$ is the weak limit of a sequence $\left\{\mu_n\in \Kc_n \right\}_{n \in \Nb }$. Then for each $n \in \Nb$, there exists $g_n\in G$ with 
$$\frac{\ell_\sigma(g_n)}{\ell_\rho(g_n)} \ge L - \frac{1}{n}$$ 
so that $\mu$ is the weak limit of $\mu_{[g_n]}$.

We claim that $\mu\in \Mc(\lambda)$. 
Suppose to the contrary that $\pi (\supp \mu)$ is not contained in $\lambda$. Then there exists a compact subset $V$ of $X_\rho-\lambda$ with non-empty interior $V^{\circ}$
so that $\mu(\pi^{-1}(V^{\circ}))>0$.  By Theorem \ref{MaximalGK}, there exists $L_0<L$ so that the restriction of $f$ to $V$ is locally $L_0$-Lipschitz. 
There exists $\delta, N>0$ 
so that if $n\ge N$, then $\mu_{[g_n]}(\pi^{-1}(V))\ge\delta$. Therefore, 
we have for all $n \ge N$ that 
$$\ell_\sigma(g_n) \le \ell_{X_\sigma}(f(g_n^*))\le \big( L(1-\delta)+L_0\delta \big)\ell_\rho(g_n)$$
where $g_n^* \subset X_{\rho}$ is the closed geodesic given by the conjugacy class $[g_n]$ and $\ell_{X_{\sigma}}(\cdot)$ denotes the length of the curve in $X_{\sigma}$.
This implies that
 $$ L(1-\delta)+L_0\delta \ge L-\frac{1}{n}\quad \text{for all }\  n\ge N,$$
 which is a contradiction.

It then follow from the upper semicontinuity of  the metric entropy function $\Mc(\rho) \to \Rb$,  see  \cite{Newhouse_entropy}, that
$$\begin{aligned}
\lim_{\epsilon\to 0}\limsup_{T \to + \infty} & \frac{\log \# \left\{ [g] \in [G]_{\rm prim} :  \frac{\ell_\sigma(g)}{\ell_\rho(g)} \ge L - \epsilon, \ T - 1 \le \ell_\rho (g) \le T \right\}}{T}  
& \le \sup_{\mu \in \Mc(\lambda)} h(\mu).
\end{aligned}
$$

Our result then follows from Lemma \ref{lem:Quint} below, which is 
stated in \cite[Lemma 3.1.1]{Quint_divergence} under the assumption that $\limsup_{T \to + \infty} \frac{ \log m([T-1, T])}{T} > 0$. However,  
its proof essentially works without strict positivity, but under non-negativity. 
The non-negativity assumption is immediately satisfied in our setting, since  we are applying the lemma to a counting measure and that, by definition,  for all $\epsilon>0$, 
there exists $g\in G$ so that $\frac{\ell_\sigma(g)}{\ell_\rho(g)}\ge L-\epsilon$.
\end{proof}

\begin{lemma}[{\cite[Lemma 3.1.1]{Quint_divergence}}] \label{lem:Quint}
    Let $m$ be a Radon measure on $\Rb_{\ge 0}$. If $\limsup_{T \to + \infty} \frac{ \log m([T-1, T])}{T} \ge 0$, then 
    $$
    \limsup_{T \to + \infty} \frac{ \log m([T-1, T])}{T} = \limsup_{T \to + \infty} \frac{ \log m([0, T])}{T}.
    $$

\end{lemma}

\begin{remark}
\label{currents proof} Here we explain how to use the theory of geodesic currents (see Section \ref{subsec:currents}) 
to rephrase the proof of Theorem \ref{general entropy}, and hence Theorem \ref{thm:subexp},
when $d=2$,  $G=\pi_1(S)$ where $S$ is a closed connected orientable surface of genus at least two.
Bonahon \cite{bonahon} showed that the geometric intersection 
number between free homotopy classes of closed curves on $S$ continuously extends to a symmetric bi-linear form $i : \Cc(S) \times \Cc(S) \to \Rb_{\ge 0}$. 
Moreover, for each $[\rho] \in \Tc(S)$, there exists $\Lc_{\rho} \in \Cc(S)$, called the Liouville current, so that 
$$i(\nu_{[g]}, \Lc_{\rho}) = \ell_{\rho}(g) \quad \text{for all }g \in \pi_1(S)
$$
where $\nu_{[g]}$ is as in Equation \eqref{eqn:current for curve}.

The closed set $\Kc_n\subset\Mc(\rho)$ in the proof of Theorem \ref{general entropy}, Equation \eqref{eqn:closedset}, can then be replaced with 
$$\widehat \Kc_n := \overline{\left\{ \frac{\nu_{[g]} }{\ell_\rho(g)} \in\Cc(S): [g]\in [G]_{\mathrm{prim}}\text{ and } \frac{i(\nu_{[g]},\Lc_\sigma)}{i(\nu_{[g]},\Lc_\rho)}\ge L_{\rho\sigma}-\frac{1}{n}\right\} }\subset\Cc(S).$$
Notice that $\widehat\Kc_n=\phi_\rho(\Kc_n)$ for all $n \in \Nb$.
\end{remark}

\section{Vanishing extremal entropy and growth indicator}

In this section, we deduce Theorem \ref{thm:subexp}, Corollary \ref{cor:GI}, and Theorem \ref{thm:necessary vanishing} from the following.

\begin{corollary} \label{cor:zero extremal}
 Suppose that $G$ is a non-elementary, torsion-free, 
Gromov hyperbolic group and $\rho : G \to \PO(2, 1)$ and 
$\sigma : G \to \PO(d, 1)$ are 
  convex cocompact representations. If $L_{\rho \sigma} >~1$, then 
$$
\Ec_{\rho \sigma}(L_{\rho \sigma}) = 0.
$$

\end{corollary}

\begin{proof}
  First notice that $\Ec_{\rho \sigma}(L_{\rho \sigma}) \ge 0$. 
By Theorem \ref{general entropy}, we have
$$
\Ec_{\rho \sigma}(L_{\rho \sigma}) \le \sup_{\mu \in \Mc(\lambda_{\rho \sigma})} h(\mu).
$$
By Theorem \ref{thm:variational classic}, the right hand side is equal to the topological entropy of the geodesic flow on $\lambda_{\rho \sigma}$, which is a geodesic lamination on the convex core of a convex cocompact hyperbolic surface $X_{\rho}$. By a result of Fathi \cite[Lemma 3.3]{Fathi_expansiveness}, this topological entropy is zero,
completing  the proof.
\end{proof}

\begin{remark}
The main obstacle to extending Corollary \ref{cor:zero extremal} to a general representation $\rho : G \to \mathsf{PO}(d,1)$ is that we do not know how to determine the topological entropy of
a geodesic lamination in higher dimensions.
\end{remark}

We now deduce our claim about the vanishing of the growth indicator. 
Recall the definition of the growth indicator from Equation \eqref{eqn:GI}, whose definition works verbatim for products of representations into $\PO(d, 1)$.

\begin{corollary} \label{cor:necessary vanishing general}
Suppose that $G$ is a non-elementary, torsion-free, 
 Gromov hyperbolic group and $\rho : G \to \PO(2, 1)$ 
  and 
 $\sigma : G \to \PO(d, 1)$ are 
 convex cocompact representations, with Zariski dense images. 
 For $\Gamma := (\rho \times \sigma)(G)$, if $\partial_{\max}$ has slope strictly bigger than $1$, then
  $$
  \psi_{\Gamma} = 0 \quad \text{on } \partial_{\max}.
  $$

\end{corollary}

\begin{proof}

Let $u = (1, L_{\rho \sigma}) \in \partial_{\max}$. It suffices to show that $\psi_{\Gamma}(u) = 0$.
Our assumptions imply that $\Gamma$ is Zariski dense, and hence we have $\psi_{\Gamma}(u) \ge 0$ by Quint \cite{Quint_divergence}. On the other hand, 
by \cite[Theorem 8.1]{chow2023jordan}, $\psi_{\Gamma}(u) \le \Ec_{\rho \sigma}(1)$. Now applying Corollary \ref{cor:zero extremal} finishes the proof.
\end{proof}

We now deduce Theorem \ref{thm:subexp}, Corollary \ref{cor:GI}, and Theorem \ref{thm:necessary vanishing}. Let $S$ be a closed surface of genus at least two.

\bigskip\noindent
 {\bf Theorem \ref{thm:subexp}.}
 {\em
If $[\rho]\ne[\sigma]\in \Tc(S)$, then
\begin{enumerate}
    \item almost maximally stretched closed geodesics have subexponential growth:
    $$
\lim_{\epsilon \to 0} \limsup_{T \to + \infty} \frac{\log \# \left\{ [g] \in [\pi_1(S)] : \frac{\ell_\sigma(g)}{\ell_\rho(g)} \ge L_{\rho\sigma} - \epsilon \text{ and } \ell_\rho(g) \le T \right\}}{T} = 0.
$$
\item almost minimally stretched closed geodesics have subexponential growth:
    $$
\lim_{\epsilon \to 0} \limsup_{T \to + \infty} \frac{\log \# \left\{ [g] \in [\pi_1(S)] : \frac{\ell_\sigma(g)}{\ell_\rho(g)} \le \frac{1}{L_{\sigma\rho}} + \epsilon \text{ and } \ell_\sigma(g) \le T \right\}}{T} = 0.
$$
\end{enumerate}
}

\bigskip\noindent
{\em Proof of  Theorem \ref{thm:subexp}:}
Since Item (2) in the statement follows from Item (1) by switching $\rho$ and $\sigma$, it suffices to prove Item (1).
In this case, Thurston \cite{thurston1998minimal} showed that $L_{\rho\sigma}>1$, so this follows from Corollary \ref{cor:zero extremal}.
\qed

\bigskip\noindent
{\bf Corollary \ref{cor:GI}.} {\em
If $[\rho]\ne[\sigma]\in\Tc(S)$ and 
$\Ga := (\rho \times \sigma)\big(\pi_1(S)\big) \subset \mathsf{PO}(2,1)\times\mathsf{PO}(2,1)$, then 
    $$
    \psi_{\Ga} = 0 \quad \text{on } \partial \Bc_{\Ga}.
    $$
}

\bigskip\noindent
{\em Proof of Corollary \ref{cor:GI}:}
Recall that by Thurston \cite{thurston1998minimal}, we have $L_{\rho \sigma} > 1$, and hence $\psi_{\Gamma} = 0$ on $\partial_{\max}$ by Corollary \ref{cor:necessary vanishing general}. Switching $\rho$ and $\sigma$ and applying the same argument, it follows that $\psi_{\Gamma} = 0$ on $\partial_{\min}$, as desired.
\qed

\bigskip\noindent
{\bf Theorem \ref{thm:necessary vanishing}.} {\em
Let  $\Gamma \subset \PO(2, 1) \times \PO(2, 1)$ be a Zariski dense Borel Anosov subgroup. Then $\psi_{\Gamma} = 0$ on at least one component of $\partial \Bc_{\Gamma}$. More precisely:
\begin{enumerate}
  \item If the slope of $\partial_{\max}$ is strictly bigger than $1$, then 
  $$
  \psi_{\Gamma}= 0 \quad \text{on } \partial_{\max}.
  $$
  \item If the slope of $\partial_{\min}$ is strictly smaller than $1$, then 
  $$
  \psi_{\Gamma}= 0 \quad \text{on } \partial_{\min}.
  $$
\end{enumerate}
}

\bigskip\noindent
{\em Proof of Theorem \ref{thm:necessary vanishing}:} 
Recall that $\Gamma$ has a finite index subgroup of the form $
(\rho \times \sigma)(G)$, for some non-elementary, torsion-free, Gromov hyperbolic group $G$ 
and non-conjugate convex cocompact representations $\rho, \sigma : G \to \PO(2, 1)$. 
Notice that $\psi_\Gamma=\psi_{(\rho\times\sigma)(G)}$, so
we may assume that $\Gamma = (\rho \times \sigma)(G)$. Now Item (1) is a special case 
of Corollary \ref{cor:necessary vanishing general}, and  Item (2) follows from Item (1) 
by switching the order of $\rho$ and $\sigma$. Finally, as $\rho$ and $\sigma$ are 
non-conjugate,  at least one of Item (1) and Item (2) occurs.
\qed

\section{Mixed behavior of extremal entropies}

In this section, we construct explicit examples showing that the strict inequality for the slopes of $\partial_{\max}$ or $\partial_{\min}$ is necessary to have the vanishing of growth indicator. In other words, we prove Proposition \ref{free examples} and Proposition~\ref{QF examples}.

\subsection{Results in $\mathsf{PO}(2,1)$}

We first consider products of representations into $\PO(2, 1)$. We show that any real number in $(0, 1)$ can be realized as the extremal entropy when the slope of $\partial_{\max}$ is one. 
Recall that $\Sigma_{g, b}$ denotes the compact connected orientable surface of genus $g$ with $b$ boundary components.

\bigskip\noindent
{\bf Proposition \ref{free examples}.} {\em 
Suppose $2g + b > 3$  and $b>0$. Then for any $\delta \in (0, 1)$, there exist non-conjugate convex cocompact representations 
$\rho, \sigma : \pi_1(\Sigma_{g, b}) \to \PO(2, 1)$  which uniformize $\Sigma_{g,b}$ so that:
\begin{enumerate}
  \item $L_{\rho \sigma} = 1$.
  \item $\Ec_{\rho \sigma}(1) = \delta$.
  \item If $\Ga := (\rho \times \sigma)(\pi_1(\Sigma_{g, b}))$, then $\partial_{\max} = \{(t, t) : t > 0 \}$ has slope 1 and 
  $$\psi_{\Gamma}(t, t) = \delta \cdot t \quad \text{for } t > 0.$$
\end{enumerate}
} 

\bigskip\noindent
{\em Proof of Proposition \ref{free examples}:} 
For simplicity, let $\Sigma := \Sigma_{g, b}$ in the proof. There exists a pair of pants $P \subset \Sigma$ such that $\partial P$ contains at least one component of $\partial \Sigma$ and $\Sigma - P$ is connected and non-empty 
  (e.g. Figure \ref{fig:1Lip}). 
  We set $C := \partial P \cap \Sigma^{\circ}$, where $\Sigma^{\circ}$ is the interior of $\Sigma$.

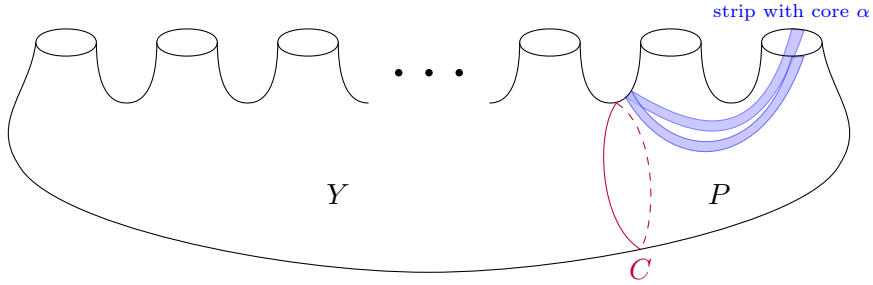
\begin{figure}[h]
\begin{tikzpicture}[scale=0.8]

\draw[blue!70, fill=blue!30, opacity=0.7] (-1.15, 0.2) .. controls (0, -0.5) and (1, -0.7) .. (1.5, 1.23) -- (1.7, 0.77+0.43) .. controls (1, -0.9) and (0, -0.7) .. (-1.25, 0.1);

\draw[blue!80, fill=blue!30, opacity=0.7] (-1.15, 0.2) .. controls (-0.5, -1) and (1, -1) .. (1.5, 0.76) -- (1.7, 0.77) .. controls (1, -1.2) and (-0.5, -1.2) .. (-1.25, 0.1);

\draw[blue] (1.5, 1.5) node {\tiny strip with core $\alpha$};

  \draw (1, 1) .. controls (1, 1.3) and (2, 1.3) .. (2, 1) .. controls (2, 0.7) and (1, 0.7) .. (1, 1);

  \draw (1, 1) .. controls (1, 0) and (0.6, 0) .. (0.5, 0) .. controls (0.4, 0) and (0, 0) .. (0, 1);

  \draw (-1, 1) .. controls (-1, 1.3) and (0, 1.3) .. (0, 1) .. controls (0, 0.7) and (-1, 0.7) .. (-1, 1);

  \draw (1 -2, 1) .. controls (1 -2, 0) and (0.6 -2, 0) .. (0.5 -2, 0) .. controls (0.4 -2, 0) and (0 -2, 0) .. (0 -2, 1);

  \draw (-1 -2, 1) .. controls (-1 -2, 1.3) and (0 -2, 1.3) .. (0 -2, 1) .. controls (0 -2, 0.7) and (-1 -2, 0.7) .. (-1 -2, 1);

  \draw (1 -4, 1) .. controls (1 -4, 0) and (0.6 -4, 0) .. (0.5 -4, 0);

  \filldraw (-4.5, 0.5) circle(1.5pt);

  \filldraw (-5, 0.5) circle(1.5pt);

  \filldraw (-4, 0.5) circle(1.5pt);

  \begin{scope}[shift={(-8, 0)}]

      \draw (2, 1) .. controls (2, 0) and (2.4, 0) .. (2.5, 0);

  \draw (1, 1) .. controls (1, 1.3) and (2, 1.3) .. (2, 1) .. controls (2, 0.7) and (1, 0.7) .. (1, 1);

  \draw (1, 1) .. controls (1, 0) and (0.6, 0) .. (0.5, 0) .. controls (0.4, 0) and (0, 0) .. (0, 1);

  \draw (-1, 1) .. controls (-1, 1.3) and (0, 1.3) .. (0, 1) .. controls (0, 0.7) and (-1, 0.7) .. (-1, 1);

  \draw (1 -2, 1) .. controls (1 -2, 0) and (0.6 -2, 0) .. (0.5 -2, 0) .. controls (0.4 -2, 0) and (0 -2, 0) .. (0 -2, 1);

  \draw (-1 -2, 1) .. controls (-1 -2, 1.3) and (0 -2, 1.3) .. (0 -2, 1) .. controls (0 -2, 0.7) and (-1 -2, 0.7) .. (-1 -2, 1);

  \end{scope}

  \draw (-11, 1) .. controls (-11.1, 0.2) and (-11.75, -0.3) .. (-11.3, -1) .. controls (-10.8, -2.0) and (-7.0, -2.8) .. (-4.5, -2.8) .. controls (-2.0, -2.8) and (1.8, -2.0) .. (2.3, -1) .. controls (2.75, -0.3) and (2.1, 0.2) .. (2, 1);

  \draw[purple] (-1.4, 0) .. controls (-1.75, -0.45) and (-1.7, -2) .. (-1, -2.42);

  \draw[purple, dashed] (-1.4, 0) .. controls (-0.7, -0.45) and (-0.75, -2) .. (-1, -2.42);

  \draw[purple] (-1, -2.42) node[below] {$C$};

  \draw (0.3, -1.5) node {$P$};

  \draw (-6, -1.5) node{$Y$};

\end{tikzpicture}
\caption{The surface $\Sigma$} \label{fig:1Lip}
\end{figure}

Extend $\partial \Sigma \cup\partial P$ to a  pants decomposition of $\Sigma$, i.e. a maximal collection of disjoint, non-parallel, and homotopically non-trivial simple closed curves. 
 Recall that this pants decomposition gives rise to Fenchel--Nielsen coordinates for the space of all marked hyperbolic structures with geodesic boundary on $\Sigma$. 

Fix $\delta \in (0, 1)$. 
We begin with a compact hyperbolic structure $Y$ on $\overline{\Sigma  -  P}$  with geodesic boundary so that if $\rho_0 : \pi_1(\overline{\Sigma - P}) \to \PO(2, 1)$ is a corresponding convex cocompact representation (i.e. $Y$ is the convex core of $X_{\rho_0}$), then
\begin{equation} \label{eqn:prescribed entropy}
\delta = \lim_{T \to + \infty} \frac{ \log \# \{ [g] \in [\pi_1(\overline{\Sigma - P})] : \ell_{\rho_0}(g) \le T \}}{T}.
\end{equation}
The existence of $\rho_0$ follows from the works of Furusawa \cite[Theorem 2]{Furusawa_dimension} and Patterson \cite[Theorem 4.1]{Patterson_fuchsian}, and Equation \eqref{eqn:Margulis}. Furusawa showed that  any real number in $(0, 1)$ can be realized as the Hausdorff dimension of the limit set of a convex cocompact Fuchsian group associated to any topological type of surface. Patterson's result, together with Equation \eqref{eqn:Margulis}, identifies this Hausdorff dimension with the topological entropy.

We then extend the hyperbolic structure on $Y$ to $\Sigma$ in two different ways and obtain the desired representations. First, we form a hyperbolic structure $Z_1$ 
on $\Sigma$ by attaching a hyperbolic structure on $P$ in which each component of  $C$ has the same length as in $Y$, without any twist along $C$. 
This gives a convex cocompact representation $\rho : \pi_1(\Sigma) \to \PO(2, 1)$ whose restriction on $\pi_1(\overline{\Sigma - P})$ is $\rho_0$, so that $Z_1$ 
is the convex core of $X_{\rho}$. Now fix a geodesic arc $\alpha \subset P$ whose endpoints belong to precisely one component of $\partial P \cap \partial \Sigma$. 
We then form a new hyperbolic sructure $Z_2$ on $\Sigma$ by doing a strip deformation along $\alpha$, i.e. removing a strip in $P$ with geodesic boundary 
whose core arc is $\alpha$, and stitching the boundary of the removed strip (see Figure \ref{fig:1Lip}). Then $Z_2$ also gives a convex cocompact representation 
$\sigma : \pi_1(\Sigma) \to \PO(2, 1)$, whose restriction on $\pi_1(\overline{\Sigma - P})$ is $\rho_0$, and $Z_2$ is the convex core of $X_{\sigma}$ as well. 
We further have a 1-Lipschitz map $f_{\rho \sigma} : X_{\rho} \to X_{\sigma}$ which is identity on $Y$ and is in the homotopy class
of $h_{\sigma} \circ h_{\rho}^{-1}$ (cf. \cite[Lemma 3.4]{thurston1998minimal}). Therefore $L_{\rho \sigma} = 1$, showing Item~(1).

To see Item (2), notice that
$$
\delta \le \Ec_{\rho \sigma}(1)
$$
by Equation \eqref{eqn:prescribed entropy} and that $\rho$ and $\sigma$ are restricted to $\rho_0$ on $\pi_1(\overline{\Sigma - P})$. Hence, it suffices to show the reverse inequality. As in the proof of Theorem \ref{general entropy}, there exists a sequence $\{ g_j \in \pi_1(\Sigma)  \}_{j \in \Nb}$ such that
\begin{itemize}
  \item for each $j \in \Nb$,
  \begin{equation} \label{eqn:sequencegj}
\frac{\ell_{\sigma}(g_j)}{\ell_{\rho}(g_j)} \ge 1 - \frac{1}{j}
  \end{equation}
and
\item the sequence $\left\{\mu_{[g_j]} \in \Mc(\rho)  \right\}_{j \in \Nb}$  has the weak limit $\mu \in \Mc(\rho)$ so that
\begin{equation} \label{eqn:ext by hmu}
\Ec_{\rho \sigma}(1) \le h(\mu).
\end{equation}
\end{itemize}

Hence, it suffices to show that $h(\mu) \le \delta$. 
One may observe, just as 
in the proof of \cite[Proposition 3.1(1)]{DGK_Margulis}, that there exists $\theta>0$  so that each time
the geodesic representative 
$g_j^*$ of  $[g_j]$ in $Z_1$ intersects $\alpha$, the angle of intersection lies in $(\theta, \pi - \theta)$.
It then follows from  \cite[Lemma 2.2]{DGK_Margulis} that there exists $w > 0$ so that
\begin{equation} \label{eqn:intersection}
\ell_{\sigma}(g_j) \le \ell_{\rho}(g_j) - w \cdot i (g_j , \alpha) \quad \text{for all } j \in \Nb
\end{equation}
where $i(g_j,\alpha)$   is the geometric intersection number, i.e. the minimal number of intersection points of a curve in the free homotopy class determined by $g_j$ with
the arc $\alpha$.

If $D\Sigma$ denotes the double of $\Sigma$, then  $D\Sigma$ is a closed orientable
surface of genus at least two. We may identify $\Sigma$ with one of the components
of $\Sigma$ in $D\Sigma$. Then the inclusion of $\Sigma$ into $D\Sigma$ induces an inclusion of $\partial\pi_1(\Sigma)$
into $\partial\pi_1(D\Sigma)$. This inclusion gives rise to an embedding
$$\psi:\Cc(\Sigma)\to \Cc(D\Sigma).$$ 

The  geometric intersection number of pairs  of closed curves on $\Sigma$ continuously extends to 
$\Cc(\Sigma)\times\Cc(\Sigma)$
by letting 
$$i(\nu,\eta):=i(\psi(\nu),\psi(\eta)).$$
(Recall that Bonahon \cite{bonahon} proved that geometric  intersection number extends continuously on $\Cc(D\Sigma)$.)
Notice that the double $D\alpha$ of $\alpha$ is a homotopically non-trivial closed cuve in $D\Sigma$ and we can define
$$i(\eta,\alpha):=i(\psi(\eta),\nu_{[D\alpha]})\quad\text{for all } \eta\in\Cc(\Sigma).$$
Notice that this is a continuous  function such that $i(\nu_{[g]},\alpha)$ is the geometric intersection number $i(g,\alpha)$ for all non-trivial
$g\in\pi_1(\Sigma)$.

Since $\phi_\rho(\mu_{[g_j]})=\frac{\nu_{[g_j]}}{\ell_\rho(g_j)}$ for all $j \in \Nb$, we see that
$\nu:=\phi_\rho(\mu)$ is a weak limit of  $\left\{\frac{\nu_{[g_j]}}{\ell_\rho(g_j)}\right\}_{j \in \Nb}$ in $\Cc(\Sigma)$.  
Since $i(\cdot,\alpha)$ is a continuous linear  function on $\Cc(\Sigma)$,
Equations \eqref{eqn:sequencegj} and \eqref{eqn:intersection} imply that
$$
i(\nu, \alpha) = 0.
$$
Since every complete geodesic in $Z_1$ is either contained in $Y \cup \partial \Sigma$ or intersects $\alpha$, 
this implies that, under the basepoint projection $T^1 X_{\rho} \to X_{\rho}$, $\supp \mu$ projects into $Y \cup \partial \Sigma$. 
Notice also that the topological entropy of the geodesic flow on a closed geodesic is zero. 
Hence, by the Variational principle (Theorem \ref{thm:variational classic}), $h(\mu)$ is bounded from above by the topological entropy of the geodesic flow on $Y$. 
By the choice of $Y$ in Equation~\eqref{eqn:prescribed entropy} and by Equation~\eqref{eqn:Margulis}, 
$$
h(\mu) \le \delta.
$$
Together with Equation \eqref{eqn:ext by hmu}, Item (2) follows.

Now set $\Gamma := (\rho \times \sigma)(\pi_1(\Sigma)) \subset \PO(2, 1) \times \PO(2, 1)$. It follows from Item (1) that $\partial_{\max} = \{ (t, t) : t >  0 \}$. 
Since $\psi_{\Gamma}$ is positively homogeneous of degree one, it suffices to show $\psi_{\Gamma}(1, 1) = \delta$. 
Since both $\rho$ and $\sigma$ are restricted to $\rho_0$ on the subgroup $\pi_1(\overline{\Sigma - P}) \subset \pi_1(\Sigma)$, we have $\delta \le \psi_{\Gamma}(1, 1)$. 
The reverse inequality follows from Item (2)  and \cite[Theorem 8.1]{chow2023jordan}. This establishes Item~(3), completing the proof.
\qed

\subsection{Results in $\mathsf{PO}(3,1)$}

We now consider quasifuchsian representations of a closed surface group and their product representation. We prove Proposition \ref{QF examples} which we restate below.

\bigskip\noindent
{\bf Proposition \ref{QF examples}.} {\em Let $S$ be a closed surface of genus at least two.
\begin{enumerate}
\item 
There exists a Fuchsian representation $\rho_1:\pi_1(S)\to\mathsf{PO}(2,1)\subset \mathsf{PO}(3,1)$ and a quasifuchsian, but not Fuchsian,
representation $\sigma_1:\pi_1(S)\to \mathsf{PO}(3,1)$, so that for $\Gamma_1 := (\rho_1 \times \sigma_1)(\pi_1(S))$,
$$
\partial_{\max} \text{ has slope } 1 \quad \text{and} \quad \psi_{\Gamma_1} > 0 \text{ on } \partial_{\max} .
$$

\item

There exist non-conjugate quasifuchsian, but not Fuchsian, representations $\rho_2, \sigma_2 : \pi_1(S) \to \PO(3, 1)$, so that for $\Gamma_2 := (\rho_2 \times \sigma_2)(\pi_1(S))$,
$$
\partial_{\max} \text{ has slope } 1 \quad \text{and} \quad \psi_{\Gamma_2} > 0 \text{ on } \partial_{\max}.
$$

\end{enumerate}
}

\bigskip\noindent
{\em Proof of Proposition \ref{QF examples}:} Let $\sigma_0:\pi_1(S)\to \mathsf{PO}(2,1)\subset\mathsf{PO}(3,1)$ be a Fuchsian representation. 
Let $\alpha$ be a non-separating simple closed geodesic on $X_{\sigma_0}$ and 
let $\{\sigma_t : \pi_1(S) \to \PO(3, 1) \}_{t\in\mathbb R}$ be the bending deformation 
of $\sigma_0$ along $\alpha$.
Results of Bridgeman--Canary--Sambarino \cite[Theorem 1.3]{BCS_bending} 
imply that there exists $\delta >0$ so that if $0\le s<t\le\delta$, then
$$\ell_{\sigma_t}(g)\le \ell_{\sigma_s}(g)$$
for all non-trivial $g\in \pi_1(S)$ with equality if and only if the closed geodesic on 
$X_{\sigma_0}$ in the free homotopy class of $g$ is disjoint from $\alpha$.
Therefore, $L_{\sigma_s\sigma_t}=1$ and $L_{\sigma_t\sigma_s}>1$. One again easily sees that if $\Gamma_{st} := (\sigma_s\times\sigma_t)(\pi_1(S))$,
then 
$
\psi_{\Gamma_{st}}(1,1)$ is bounded from below by the topological entropy of the geodesic flow on $X_{\sigma_0} - \alpha$. In particular, $\psi_{\Gamma_{st}}(1, 1) > 0$, and hence choosing appropriate $s$ and $t$ for Item (1) and Item (2) finishes the proof.
\qed

\end{document}